\documentclass{amsart}
\usepackage{amsfonts,amssymb,amsmath,amsthm}
\usepackage{url}
\usepackage{enumerate}

\newtheorem{thrm}{Theorem}[section]
\newtheorem{lem}[thrm]{Lemma}

\newtheorem{cor}[thrm]{Corollary}
\theoremstyle{definition}
\newtheorem{definition}[thrm]{Definition}
\newtheorem{remark}[thrm]{Remark}
\numberwithin{equation}{section}

\newcommand\supp{\mathop{\rm supp}}

\author{B. Willson}
\address{
Department of Mathematics\\
Hanyang University\\
222 Wangsimni-ro, Seongdong-gu\\
Seoul, Korea}
\email{bwillson@ualberta.ca}
\thanks{The author gratefully acknowledges the financial support of Hanyang University.}

\keywords{invariant measure, hypergroup, amenable, function translations}
\subjclass{Primary 43A62, Secondary 43A05, 43A07}
\begin{document}

\title[Fixed point theorem and Haar measure for certain hypergroups]{A Fixed Point Theorem and the Existence of a Haar Measure for Hypergroups Satisfying Conditions Related to Amenability}

\begin{abstract}
In this paper we present a fixed point property for amenable hypergroups which is analogous to Rickert's fixed point theorem
for semigroups.  It equates the existence of a left invariant mean on the space of weakly right uniformly continuous functions to the existence of a fixed point for any action of the hypergroup.   
Using this fixed point property, a certain class of hypergroups are shown  to have a left Haar measure.
\end{abstract}
\maketitle

\section{Introduction}
Hypergroups arise as generalizations of the measure algebra of a locally compact group wherein the product of two points is a probability measure rather than a single point.   The formalization of hypergroups was introduced in the 1970s by Jewett \cite{jewett}, Dunkl \cite{dunkl}, and Spector \cite{spector}.  Actions of hypergroups have been considered in \cite{singh, tahmasebi}.  Amenable hypergroups have been considered in \cite{skantharajah, lasserskantharajah, willsonTAMS, alaghmandan}.  As with groups, there are connections between invariant means on function spaces and fixed points of actions of the hypergroup.

It is a longstanding problem whether or not every hypergroup has a left Haar measure.  It is known that if a hypergroup is compact, discrete, or abelian then it does admit a Haar measure.

Every locally compact group 
$G$ admits a left Haar measure $\lambda_G$.  This Haar measure can be viewed as a positive, linear, (unbounded), left translation invariant
functional on $C_C(G)$.
It follows from the existence of this functional that for positive $f\in C_C(G)$ and positive $\mu, \nu\in M(G)$ then 
\begin{equation}
\label{condition}
\mu\ast f(x)\leq \nu\ast f(x)\ \  \forall x\in G\Rightarrow \|\mu\|\leq\|\nu\|.
\end{equation}
A similar statement can be made for amenable groups and positive, continuous, bounded functions with non-zero mean value.
One might expect this statement to be true more generally, but there are examples of unbounded functions and bounded functions on non-amenable groups for which \eqref{condition} does not hold.

The main result of this paper is to show that for amenable hypergroups, the existence of a Haar measure is equivalent to this positivity of translation property \eqref{condition}.  In section 2, we provide the neccessary preliminaries.  In section 3, we prove a fixed point theorem for the action of an amenable hypergroup.  In the final section, we use property \eqref{condition} to prove the existence of a measure which is invariant for all translations of $f$.  We build on this, and use the fixed point theorem to prove the final result.

\section{Background and Definitions}
\begin{definition}
A hypergroup, $H$, is a non-empty locally compact Hausdorff topological space which 
satisfies the following conditions (see \cite{bloomheyer} for more details on hypergroups):
\begin{enumerate}
\item
There is a binary operation $\ast$, called convolution, on the vector space of bounded Radon measures
turning it into an algebra.
\item
For $x, y\in H$, the convolution of the two point measures is a  
probability measure, and $\supp(\delta_x\ast\delta_y)$ is compact.
\item
The map $(x,y)\mapsto \delta_x\ast\delta_y$ from $H\times H$ to the compactly supported probability measures on $H$ is continuous.
\item
The map $H\times H \ni (x,y)\mapsto \supp(\delta_x\ast\delta_y)$ is continuous with respect
to the Michael topology on the space of compact subsets of $H$.
\item
There is a unique element $e\in H$ such that for every $x\in H$, 
$\delta_x\ast\delta_e = \delta_e\ast\delta_x = \delta_x$.
\item
There exists a homeomorphism $\check{}:H\rightarrow H$ such that for all $x\in H$, 
$\check{\check{x}}=x$, which can be extended to $M(H)$ via $\check{\mu}(A) = \mu(\{ 
x\in H: \check{x}\in A\})$, and such that $(\mu\ast\nu)\check{}=\check{\nu}\ast\check{\mu}$.
\item
For $x,y\in H$, $e\in\supp(\delta_x\ast\delta_y)$ if and only if $y=\check{x}$.
\end{enumerate}
\end{definition}

\begin{definition}
We denote the continuous and bounded complex-valued functions on $H$ by $C(H)$.  For the compactly supported continuous functions we use $C_C(H)$.  We use $C(H)^+, C_C(H)^+$ to denote the functions in those spaces which take only non-negative real values. We denote the complex space of bounded regular Borel measures on $H$ by $M(H)$, those that are compactly supported measures by $M_C(H)$, and those that are real valued measures by $M(H,\mathbb{R})$.
\end{definition}

\begin{definition}
A left Haar measure for $H$ is a non-zero regular Borel measure (with values in $[0,\infty]$), $\lambda$ which is left-invariant in the sense that for
any $f\in C_C(H)$, we have that $\lambda(\delta_x\ast f)=\lambda(f)$ for all $x\in H$.
\end{definition}

\begin{remark}It remains an open question whether every hypergroup admits a left Haar measure.  If $H$ does admit a left Haar measure $\lambda$, 
however, it is unique up to a scalar multiple\cite{jewett}.  
For hypergroups with a left Haar measures we are able to define the standard $L^p(H)$ function spaces.
\end{remark}

\begin{definition}
We say that a continuous function $f\in C(H)$ is right uniformly continuous [weakly right uniformly continuous] if the map
\[
H\ni x\mapsto \delta_x\ast f
\]is continuous in norm [weakly].  We denote the collection of right uniformly continuous functions [weakly right uniformly continuous] on $H$ by $UCB_r(H)$ [$WUCB_r(H)$].
\end{definition}
\begin{remark}
Skantharajah \cite{skantharajah} showed that for hypergroups with left Haar measure, $UCB_r(H) = L^1(H)\ast L^\infty(H)$.
\end{remark}

\section{Fixed point property}
Let $H$ be a hypergroup.
Let $E$ be a Hausdorff locally convex vector space and let $K\subset E$ be a compact, convex subset.  
Suppose that there is a separately continuous mapping 
$\cdot:H\times K\rightarrow K$.  Then for $x, y\in H$ and
$\xi\in K$ the weak integral 
\[
\int_H (t\cdot \xi) d(\delta_x\ast\delta_y)(t) 
\]
exists uniquely in $K$.

\begin{definition}
A {\it separately continuous [jointly continuous] action} of $H$ on $K$ is a separately continuous [jointly continuous] mapping $\cdot:H\times K\rightarrow K$ such that
\begin{enumerate}
\item
$e\cdot \xi = \xi$ for all $\xi \in K$;
\item
$x\cdot(y\cdot\xi) = \int_H (t\cdot \xi) d(\delta_x\ast\delta_y)(t)$.
\end{enumerate}

Furthermore, the action is called affine if, for each $x\in H$ $\xi\mapsto x\cdot \xi$ is affine.
\end{definition}

The following theorem is similar to a result of Rickert for amenable semigroups as presented in \cite{paterson}.
\begin{thrm}
\label{thm:FPWUCB}
Suppose that $UCB_r(H)$ has a left invariant mean $m$.  Then for each jointly continuous affine action of $H$ on some
$K$, a compact convex subset of a Hausdorff locally convex vector space, there is a point $\xi_0\in K$ such that
$x\cdot \xi_0=\xi_0$ for all $x\in H$.  Furthermore, the result holds if $UCB_r(H)$ is replaced by $WUCB_r(H)$ and jointly 
continuous is replaced by
separately continuous.
\end{thrm}
\begin{proof}

Suppose that there is a hypergroup action of $H$ on $K$.  We denote the set of affine functions on $K$ by $Aff(K)$.  It is clear
that for each point $\xi\in K$, evaluation at $\xi$ is a mean on $Aff(K)$.  Indeed, $K$ can be identified with the collection of {\bf all}
means on $Aff(K)$ and this identification is an affine homeomorphism.  (See for instance \cite{paterson} 2.20)

Given this identification, we see that the existence of a fixed point in $K$ is equivalent to the existence of a mean on $Aff(K)$ which
is invariant under the action of $H$.  

Suppose that $\phi\in Aff(K)$ and $\xi_0\in K$.  Let $\hat{\xi}_0(\phi)\in C(H)$ be defined by $\hat{\xi}_1(\phi)(x)=\phi(x\cdot \xi_1)$.
It is clear that $\hat{\xi}_1(\phi)$ is a continuous function provided that that action of $H$ is (separately) continuous.  We further
claim that if the action is separately continuous then $\hat{\xi}_1(\phi)$ is in $WUCB_r(H)$ and if it is jointly continuous then 
$\hat{\xi}_1(\phi)$
is in $UCB_r(H)$.

To show the first claim, consider some mean $F$ on $C(H)$ and some net $x_\alpha\rightarrow x$ in $H$.  Then 
$F\circ \hat{\xi}_1: Aff(K)\rightarrow \mathbb{C}$ is a mean on $Aff(K)$.  It follows then, that there is some $\xi_2\in K$
such that $F\circ \hat{\xi}_1(\phi) = \phi(\xi_2)$.  Therefore, 

\begin{align*}
\langle F,  \delta{\check{x_\alpha}}\ast\hat{\xi}_1(\phi)\rangle  &= \langle F\circ \hat{\xi}_1, \delta{\check{x_\alpha}}\ast\phi\rangle\\
&=\delta{\check{x_\alpha}}\ast\phi(\xi_2)\\
&=\phi(x_\alpha\cdot\xi_2)\\
&\rightarrow \phi(x\cdot\xi_2)\\
&=\langle F,  \delta{\check{x}}\ast\hat{\xi}_1(\phi)\rangle
\end{align*}

From this we conclude that for any $F\in C(H)^*$ the same holds.  That is, that $\hat{\xi}_0(\phi)$ is in $WUCB_r(H)$.
Now since there is a left invariant mean $M$ on $WUCB_r(H)$ it follows that $M\circ\hat{\xi}_1$ is a mean on $Aff(K)$ which
corresponds to evaluation at some point $\xi_0\in K$.  It is apparent that $\xi_0$ is a fixed point of the action of $H$.

For the second claim, observe that $\|\delta_{\check{x_\alpha}}\ast\phi-\delta_{\check{x}}\ast\phi\|\rightarrow 0$ as 
$x_\alpha\rightarrow x$ 
since the action is jointly continuous.  Furthermore, since $\hat\xi$ is contractive for each $\xi\in K$, 
$\|\delta_{\check{x_\alpha}}\ast\hat{\xi}_1(\phi)-\delta_{\check{x}}\ast\hat{\xi}_1(\phi)\|\rightarrow 0$ as $x_\alpha\rightarrow x$.
But this shows that $\hat{\xi}_1(\phi)$ is in $UCB_r(H)$ as required.  By a similar argument as above, the mean on $UCB_r(H)$ generates a 
fixed point in $K$.

\end{proof}

\section{Existence of a left Haar measure}

The existence of a Haar measure on an arbitrary hypergroup is still an open question.  In this section we present an 
approach motivated in part by a result of Izzo\cite{izzo} which uses the Markov-Kakutani fixed point theorem to prove the existence of a 
Haar measure on abelian groups.  We use the fixed point theorem  \eqref{thm:FPWUCB}
from the previous section to show that 
amenable hypergroups which satisfy an additional positivity property of translations have a left Haar measure.  This property is related to amenability in the following sense.

\begin{remark}

Every locally compact group $G$ admits a left Haar measure.  Some of the key properties of the Haar measure are as follows.  For any $f, g\in C_C(G)$ and $\mu\in M(G)$ we have
\begin{align*}
f(x)\leq g(x)\forall x\in G &\Rightarrow \int_G f(x)\, d\lambda(x)\leq\int_G g(x)\, d\lambda(x), \text{ and}\\
\int_G (\mu\ast f)(x)\, d\lambda(x) &=\mu(H) \int_G f(x)\, d\lambda(x).
\end{align*}
Hence for every non-zero $f\in C_C(G)^+$ and $\mu, \nu\in M(G)^+$ if $\mu\ast f\leq \nu\ast f$ then $\|\mu\|\leq\|\nu\|$.

Relatedly, if $G$ is amenable with left invariant mean $m\in C(G)^*$ then for every $f\in C(G)^+$ with $m(f)>0$ and $\mu, \nu\in M(G)^+$ if $\mu\ast f\leq \nu\ast f$ then $\|\mu\|\leq\|\nu\|$.
\end{remark}

\begin{definition}
Let $H$ be a hypergroup.  Fix some non-zero $f\in C_C(H)^+$.  We say that $H$ has the positivity property of translations of $f$ if 
for every $\mu, \nu\in M^+(H)$ 
\begin{equation}
\label{ppt}
\mu\ast f\leq \nu\ast f\Rightarrow \|\mu\|\leq\|\nu\|.
\end{equation}

\end{definition}

\begin{lem}
\label{lemma:ppt}
Suppose that $H$ is a hypergroup with property \eqref{ppt} for some non-zero $f\in C_C(H)^+$.

Then there is a positive, linear functional $\Gamma$ from the real-vector space $C_C(H, \mathbb{R})$ to $\mathbb{R}$ such that $\Gamma(f) = 1$ and $\Gamma (\rho\ast f) = \rho(H)$ for $\rho\in M_C(H, \mathbb{R})$.  
\end{lem}

\begin{proof}
Let $V_f:=\{\mu\ast f: \mu\in M_C(H, \mathbb{R})\}$.  Since $M_C(H, \mathbb{R})$ is a vector space, it follows that $V_f$ is also a vector space.  Define the linear functional $\Gamma_f$ on $V_f$ by
\[
\Gamma_f(\mu\ast f) = \mu(H).
\] 

Claim: This map is well defined.  

If $\mu\ast f=\nu\ast f$ then 
consider a Jordan decomposition, 
$(\mu_+ - \mu_-)\ast f = (\nu_+ - \nu_-)\ast f$.
 Then $(\mu+\mu_-+\nu_-)\ast f = (\nu+\nu_-+\mu_-)\ast f$ and these are both positive with equal norm by property \eqref{ppt} hence
$\mu(H) = \nu(H)$.  
Therefore the map $\Gamma_f$ is well defined.

Claim: This map is positive.

By a similar argument, if $0\leq \nu\ast f$ then $\nu_-\ast f\leq \nu_+\ast f$ hence $\nu_-(H)\leq \nu_+(H)$ and so $\nu(H) = \nu_+(H)-\nu_-(H)$ is positive.

By \cite{bloomheyer}[Lemma 1.2.22] for any $g\in C_C(H, \mathbb{R})$ there is a $\mu\in M(H, \mathbb{R})_C^+$ such that $g\leq \mu\ast f$.  Hence we can apply M. Riesz's extenstion theorem to extend $\Gamma_f$ to a positive linear functional on all of $C_C(H, \mathbb{R})$.
\end{proof}
  
The above lemma shows that the collection of positive linear functionals on $C_C(H)$ sending all translations of the function $f$ by elements of $H$ to $1$ is non-empty.  By applying the fixed point theorem to this collection, we prove the existence of a non-zero positive linear functional on $C_C(H)$ which is fixed under the action of translation by elements of the hypergroup.  This non-zero positive linear functional is precisely a Haar measure.

\begin{cor}
Suppose that $H$ has property \eqref{ppt} for some $f$ 
as above.  Then 
the collection $K$ of all positive linear functionals $\Lambda$ on $C_C(H)$ satisfying:
\[
\Lambda(\mu\ast f)= \mu(H) \text{ for every }\mu\in M_{C}(H).
\]
is non-empty.
\end{cor}
\begin{proof}
It suffices to use the complexification of $\Gamma$ from lemma \eqref{lemma:ppt}.
\end{proof}

\begin{thrm}
Suppose $H$ has property \eqref{ppt}.  Suppose also that $WUCB_r(H)$ has a left invariant mean.  Then $H$ has a left Haar measure.
\end{thrm}
\begin{proof}

Let $K$ be the collection of all positive linear functionals on $C_C(H)$ satisfying:
\[
\Lambda(\mu\ast f)= \mu(H)\text{ for every }\mu\in M_{C}(H).
\]

We equip $K$ with the weak* topology induced by $C_C(H)$.  If $\Lambda_\alpha
\rightarrow \Lambda$ and each $\Lambda_\alpha\in K$ then 
$\Lambda$ is in $K$ and 
so $K$ is closed.
For $g\in C_C(H)$ there is a $\mu_g\in M(H)^+$ such that $|g|\leq \mu_g\ast f$.
Subsequently for any $\Lambda\in K$, $|\Lambda(g)|\leq\Lambda(|g|)\leq \Lambda(\mu_g\ast f)\leq \|\mu_g\|$.  
Hence $K$ is compact.  

To see that $K$ is convex, suppose that $\Lambda_1, \Lambda_2\in K$ and $0\leq c\leq 1$. 
Then
$(c\Lambda_1 +(1-c)\Lambda_2)(\mu\ast f)= c+ 1-c = 1$ so $(c\Lambda_1 +(1-c)\Lambda_2)$
is an element of 
$K$.

It now suffices to show that left translation is an action of $H$ on $K$ and that this action is separately continuous.  From this we can 
apply Theorem \ref{thm:FPWUCB} to find a point in $K$ which is fixed under left translation.  

For $x\in H$, $g\in C_C(H)$ and $\Lambda\in K$ we define $x\cdot \Lambda(g) = \Lambda(\delta_{\check{x}}\ast g)$.  Then $e\cdot\Lambda=\Lambda$ and $x\cdot\Lambda(\mu\ast f) = \Lambda\left((\delta_{\check{x}}\ast\mu)\ast f\right) = 1$
so
 $K$ is closed under the action of $H$.  Furthermore
\begin{align*}
x\cdot (y\cdot\Lambda)(g)&=y\cdot\Lambda(\delta_{\check{x}}\ast g)\\
&= \Lambda(\delta_{\check{y}}\ast\delta_{\check{x}}\ast g) \\
&= \int \Lambda(\delta_{\check{t}}\ast g) d(\delta_x\ast \delta_y)(t))\\
&= \int t\cdot\Lambda(g) d(\delta_x\ast \delta_y)(t).
\end{align*}

To see that this action is separately continuous, consider a net $x_\alpha\rightarrow x\in H$. 
Eventually $x_\alpha$ will stay within a compact neighbourhood $N$ of $x$.  Then for $\Lambda\in K$ and $g\in C_C(H)$ it suffices to consider the restriction $ \Lambda_{|_{\check{N}\ast\supp(g)}}$ of $\Lambda$ which is a finite measure on the compact set $\check{N}\ast\supp(g)$.  This allows us to consider the right translate of $g$ by this measure (which is a continuous function) and we get
\begin{align*}
x_\alpha\cdot \Lambda(g) &= \Lambda(\delta_{\check{x_\alpha}}\ast g)\\
&=\Lambda_{|_{\check{N}\ast\supp(g)}}(\delta_{\check{x_\alpha}}\ast g)\\
&=\delta_{x_\alpha}(g\ast\check{\Lambda}_{|_{\check{N}\ast\supp(g)}})\\
&\rightarrow \delta_x(g\ast\check{\Lambda}_{|_{\check{N}\ast\supp(g)}})\\
&=x\cdot \Lambda(g). 
\end{align*}

Now suppose there is a net $\Lambda_\beta\rightarrow \Lambda\in K$. So for $x\in H$ and $f\in C_C(H)$
\begin{align*}
x\cdot\Lambda_\beta(f) &=\Lambda_\beta(\delta_{\check{x}}\ast f)\\
&\rightarrow \Lambda(\delta_{\check{x}}\ast f)\\
&=x\cdot \Lambda (f).
\end{align*}

Therefore the action is separately continuous.

By Theorem \ref{thm:FPWUCB} there is a fixed point $\Lambda_0$ in $K$ under the action of $H$.  This fixed point is a positive linear functional on $C_C(H)$ satisfying 
\[
\Lambda_0(g) = \Lambda_0(\delta_x\ast g) \ \ \forall g\in C_C(H), x\in H.
\]
It is non-zero since $\Lambda_0(f) = 1$ so $\Lambda_0$ is a left Haar measure for $H$.
\end{proof}
\begin{cor}
Let $H$ be a hypergroup with a left translation invariant mean on $WUCB_r(H)$.  The following are equivalent:
\begin{enumerate}
\item
There is a left Haar measure on $H$.
\item
The property of positivity of translations \eqref{ppt} holds for every $f\in C_C(H)^+$.
\item
The property of positivity of translations \eqref{ppt} holds for some nonzero $f\in C_C(H)^+$.
\end{enumerate}
\end{cor}

\proof[Acknowledgements]
Thanks to Hanyang University for providing funding in the form of a research fund for new professors.

\bibliographystyle{amsplain}

\end{document}